\documentclass[reqno]{amsart}

\usepackage[utf8]{inputenc}
\usepackage[usenames,dvipsnames]{color}
\usepackage{amsmath}
\usepackage{enumerate}
\usepackage{amssymb}
\usepackage{amsthm}
\usepackage{mathtools}
\usepackage[mathscr]{euscript}
\usepackage{tikz-cd}
\usepackage[hidelinks]{hyperref}
\usepackage{cleveref}
\usepackage{etoolbox}

\numberwithin{equation}{section}
\newtheorem{thm}[equation]{Theorem}
\newtheorem*{thm*}{Theorem}
\newtheorem*{prop*}{Proposition}
\newtheorem{prop}[equation]{Proposition}
\newtheorem{lem}[equation]{Lemma}
\newtheorem{cor}[equation]{Corollary}

\theoremstyle{remark}
\newtheorem{defn}[equation]{Definition}
\newtheorem{ex}[equation]{Example}
\newtheorem{rem}[equation]{Remark}
\newtheorem{rec}[equation]{Recollection}
\newtheorem*{conv*}{Convention}
\newtheorem*{Ack}{Acknowledgements}

\newcommand{\nc}{\newcommand}
\nc{\dmo}{\DeclareMathOperator}

\nc{\sfA}{\mathsf{A}}
\nc{\sfB}{\mathsf{B}}
\nc{\sfC}{\mathsf{C}}
\nc{\sfD}{\mathsf{D}}
\nc{\sfE}{\mathsf{E}}
\nc{\sfF}{\mathsf{F}}
\nc{\sfG}{\mathsf{G}}
\nc{\sfH}{\mathsf{H}}
\nc{\sfI}{\mathsf{I}}
\nc{\sfJ}{\mathsf{J}}
\nc{\sfK}{\mathsf{K}}
\nc{\sfL}{\mathsf{L}}
\nc{\sfM}{\mathsf{M}}
\nc{\sfN}{\mathsf{N}}
\nc{\sfO}{\mathsf{O}}
\nc{\sfP}{\mathsf{P}}
\nc{\sfQ}{\mathsf{Q}}
\nc{\sfR}{\mathsf{R}}
\nc{\sfS}{\mathsf{S}}
\nc{\sfT}{\mathsf{T}}
\nc{\sfU}{\mathsf{U}}
\nc{\sfV}{\mathsf{V}}
\nc{\sfW}{\mathsf{W}}
\nc{\sfX}{\mathsf{X}}
\nc{\sfY}{\mathsf{Y}}
\nc{\sfZ}{\mathsf{Z}}

\nc{\scA}{\mathscr{A}}
\nc{\scB}{\mathscr{B}}
\nc{\scC}{\mathscr{C}}
\nc{\scD}{\mathscr{D}}
\nc{\scE}{\mathscr{E}}
\nc{\scF}{\mathscr{F}}
\nc{\scG}{\mathscr{G}}
\nc{\scH}{\mathscr{H}}
\nc{\scI}{\mathscr{I}}
\nc{\scJ}{\mathscr{J}}
\nc{\scK}{\mathscr{K}}
\nc{\scL}{\mathscr{L}}
\nc{\scM}{\mathscr{M}}
\nc{\scN}{\mathscr{N}}
\nc{\scO}{\mathscr{O}}
\nc{\scP}{\mathscr{P}}
\nc{\scQ}{\mathscr{Q}}
\nc{\scR}{\mathscr{R}}
\nc{\scS}{\mathscr{S}}
\nc{\scT}{\mathscr{T}}
\nc{\scU}{\mathscr{U}}
\nc{\scV}{\mathscr{V}}
\nc{\scW}{\mathscr{W}}
\nc{\scX}{\mathscr{X}}
\nc{\scY}{\mathscr{Y}}
\nc{\scZ}{\mathscr{Z}}

\nc{\mcA}{\mathcal{A}}
\nc{\mcB}{\mathcal{B}}
\nc{\mcC}{\mathcal{C}}
\nc{\mcD}{\mathcal{D}}
\nc{\mcE}{\mathcal{E}}
\nc{\mcF}{\mathcal{F}}
\nc{\mcG}{\mathcal{G}}
\nc{\mcH}{\mathcal{H}}
\nc{\mcI}{\mathcal{I}}
\nc{\mcJ}{\mathcal{J}}
\nc{\mcK}{\mathcal{K}}
\nc{\mcL}{\mathcal{L}}
\nc{\mcM}{\mathcal{M}}
\nc{\mcN}{\mathcal{N}}
\nc{\mcO}{\mathcal{O}}
\nc{\mcP}{\mathcal{P}}
\nc{\mcQ}{\mathcal{Q}}
\nc{\mcR}{\mathcal{R}}
\nc{\mcS}{\mathcal{S}}
\nc{\mcT}{\mathcal{T}}
\nc{\mcU}{\mathcal{U}}
\nc{\mcV}{\mathcal{V}}
\nc{\mcW}{\mathcal{W}}
\nc{\mcX}{\mathcal{X}}
\nc{\mcY}{\mathcal{Y}}
\nc{\mcZ}{\mathcal{Z}}

\nc{\mfp}{\mathfrak{p}}
\nc{\mfq}{\mathfrak{q}}
\nc{\mfm}{\mathfrak{m}}
\nc{\mfj}{\mathfrak{j}}
\nc{\mfs}{\mathfrak{s}}
\nc{\mfh}{\mathfrak{h}}

\nc{\rmh}{\mathrm{h}}
\nc{\rmfp}{\mathrm{fp}}
\nc{\rmc}{\mathrm{c}}
\nc{\rms}{\mathrm{s}}
\nc{\rma}{\mathrm{a}}
\nc{\rmb}{\mathrm{b}}
\nc{\rml}{\mathrm{l}}
\nc{\rmL}{\mathrm{L}}
\nc{\rmC}{\mathrm{C}}
\nc{\rmD}{\mathrm{D}}
\nc{\rmop}{\mathrm{op}}
\nc{\rmS}{\mathrm{S}}
\nc{\rmK}{\mathrm{K}}
\nc{\rmN}{\mathrm{N}}
\nc{\rmH}{\mathrm{H}}
\nc{\rmZ}{\mathrm{Z}}

\nc{\bfA}{\mathbf{A}}
\nc{\bfB}{\mathbf{B}}
\nc{\bfC}{\mathbf{C}}
\nc{\bfD}{\mathbf{D}}
\nc{\bfE}{\mathbf{E}}
\nc{\bfF}{\mathbf{F}}
\nc{\bfG}{\mathbf{G}}
\nc{\bfH}{\mathbf{H}}
\nc{\bfI}{\mathbf{I}}
\nc{\bfJ}{\mathbf{J}}
\nc{\bfK}{\mathbf{K}}
\nc{\bfL}{\mathbf{L}}
\nc{\bfM}{\mathbf{M}}
\nc{\bfN}{\mathbf{N}}
\nc{\bfO}{\mathbf{O}}
\nc{\bfP}{\mathbf{P}}
\nc{\bfQ}{\mathbf{Q}}
\nc{\bfR}{\mathbf{R}}
\nc{\bfS}{\mathbf{S}}
\nc{\bfT}{\mathbf{T}}
\nc{\bfU}{\mathbf{U}}
\nc{\bfV}{\mathbf{V}}
\nc{\bfW}{\mathbf{W}}
\nc{\bfX}{\mathbf{X}}
\nc{\bfY}{\mathbf{Y}}
\nc{\bfZ}{\mathbf{Z}}

\nc{\bbN}{\mathbb{N}}
\nc{\bbZ}{\mathbb{Z}}
\nc{\bbQ}{\mathbb{Q}}

\nc{\gG}{\Gamma}
\nc{\gL}{\Lambda}
\nc{\gD}{\Delta}
\nc{\gS}{\Sigma}
\nc{\ga}{\alpha}
\nc{\gb}{\beta}
\nc{\g}{\gamma}
\nc{\gd}{\delta}
\nc{\e}{\epsilon}
\nc{\gz}{\zeta}
\nc{\gh}{\eta}
\nc{\gu}{\theta}
\nc{\gi}{\iota}
\nc{\gk}{\kappa}
\nc{\gl}{\lambda}
\nc{\gm}{\mu}
\nc{\gn}{\nu}
\nc{\gj}{\xi}
\nc{\gp}{\pi}
\nc{\gr}{\rho}
\nc{\gs}{\sigma}
\nc{\gt}{\tau}
\nc{\gf}{\phi}
\nc{\gx}{\chi}
\nc{\gc}{\psi}
\nc{\go}{\omega}

\nc{\wh}{\widehat}
\nc{\wt}{\widetilde}
\nc{\ol}{\overline}
\nc{\ul}{\underline}
\nc{\tl}{\tilde}
\nc{\ot}{\otimes}
\nc{\xr}{\xrightarrow}

\dmo{\Coloc}{Coloc}
\dmo{\coloc}{coloc}
\dmo{\Loc}{Loc}
\dmo{\loc}{loc}
\dmo{\gldim}{gldim}
\dmo{\Spec}{Spec}
\dmo{\Sing}{Sing}
\dmo{\Spc}{Spc}
\dmo{\Supp}{Supp}
\dmo{\Cosupp}{Cosupp}
\dmo{\Mod}{Mod}
\dmo{\smod}{mod}
\dmo{\sMod}{\ul{\Mod}}
\dmo{\smodu}{\ul{\smod}}
\dmo{\Proj}{Proj}
\dmo{\proj}{proj}
\dmo{\Flat}{Flat}
\dmo{\Inj}{Inj}
\dmo{\ginj}{GInj}
\dmo{\sginj}{\ul{\ginj}}
\dmo{\PInj}{PInj}
\dmo{\Coh}{Coh}
\dmo{\QCoh}{QCoh}
\dmo{\RHom}{RHom}
\dmo{\Hom}{Hom}
\dmo{\shom}{hom}
\dmo{\Ext}{Ext}
\dmo{\End}{End}
\dmo{\Ob}{Ob}
\dmo{\Ann}{Ann}
\dmo{\pd}{pd}
\dmo{\res}{res}
\dmo{\Ker}{Ker}
\dmo{\coker}{Coker}
\dmo{\im}{Im}
\dmo{\colim}{colim}
\dmo{\hocolim}{hocolim}
\dmo{\Id}{Id}

\makeatletter
\patchcmd{\@setaddresses}{\indent}{\noindent}{}{}
\patchcmd{\@setaddresses}{\indent}{\noindent}{}{}
\patchcmd{\@setaddresses}{\indent}{\noindent}{}{}
\patchcmd{\@setaddresses}{\indent}{\noindent}{}{}
\makeatother

\begin{document}
\title{Colocalizing subcategories of singularity categories}
\author{Charalampos Verasdanis}
\address{Charalampos Verasdanis \\ School of Mathematics and Statistics \\ University of Glasgow}
\email{c.verasdanis.1@research.gla.ac.uk}
\date{}
\subjclass{18F99, 18G80, 13D09}
\keywords{Singularity category, colocalizing subcategory, costratification, cosupport, relative tensor-triangular geometry}

\begin{abstract}
Utilizing previously established results concerning costratification in relative tensor-triangular geometry, we classify the colocalizing subcategories of the singularity category of a locally hypersurface ring and then we generalize this classification to singularity categories of schemes with hypersurface singularities.
\end{abstract}

\maketitle

\section{Introduction}
\label{sec:intro}%
Let $R$ be a commutative noetherian ring. The \emph{singularity category} $\rmS(R)=\rmK_{\mathrm{ac}}(\Inj R)$, i.e., the homotopy category of acyclic complexes of injective $R$-modules, was introduced and studied in~\cite{Krause05} (wherein it was called the \emph{stable derived category}), where it was shown that $\rmS(R)$ is a compactly generated triangulated category and that there is a \emph{stabilization functor} $I_\gl Q_\rho \colon \rmD(R) \to \rmS(R)$ that can be used to describe the compact objects of $\rmS(R)$. A classic problem regarding such categories is the classification of the localizing and colocalizing subcategories, e.g., for the derived category $\rmD(R)$, see~\cite{Neeman92,Neeman11}. The modern approach is via the machinery of tensor-triangular geometry~\cite{Balmer05,BalmerFavi11} and the theory of stratification and costratification~\cite{BarthelCastellanaHeardSanders23,BarthelHeardSanders23,
BensonIyengarKrause11a,BensonIyengarKrause11b,BensonIyengarKrause12,Verasdanis23a,
Verasdanis23b}, provided that the given category is tensor-triangulated. However, $\rmS(R)$ is not tensor-triangulated (at least not in any obvious way) and so a different approach is adopted. In~\cite{Stevenson13}, Stevenson developed a theory of actions of rigidly-compactly generated tensor-triangulated categories on compactly generated triangulated categories and in~\cite{Stevenson14}, he utilized the action of $\rmD(R)$ on $\rmS(R)$ to classify the localizing subcategories of $\rmS(R)$, for a locally hypersurface ring $R$ and then generalized this classification result to the singularity category $\rmS(X)$ of a noetherian separated scheme $X$ with hypersurface singularities.

In this paper we apply the theory of costratification developed in~\cite{Verasdanis23a} in order to classify the colocalizing subcategories of $\rmS(R)$, for a locally hypersurface ring $R$ and then we generalize our result to the singularity category of a scheme with hypersurface singularities. Specifically, for the case of rings, using the action of $\rmD(R)$ on $\rmS(R)$, we obtain a notion of cosupport for the objects of $\rmS(R)$. The cosupport of an object $A\in \rmS(R)$ is
\[
\Cosupp(A)=\{\mfp \in \Spec R \mid \Hom_R(g_\mfp,A)\neq 0\}\subseteq \Sing R,
\]
where $g_\mfp=K_\infty(\mfp)\ot_R R_\mfp$ is the Balmer--Favi idempotent associated with $\mfp\in \Spec R$ and $\Sing R$ is the singular locus of $R$. The assignment of cosupport allows us to define the maps
\[
\begin{tikzcd}
\Coloc(\rmS(R)) \rar["\sigma",shift left] & \scP(\Sing R){,} \lar["\tau",shift left]
\end{tikzcd}
\]
where $\sigma(\scC)=\bigcup_{A\in \scC}\Cosupp(A)$ and $\tau(W)=\{A\in \rmS(R) \mid \Cosupp(A)\subseteq W\}$. If $\sigma$ and $\tau$ are bijections inverse to each other, then we say that $\rmS(R)$ is \emph{costratified}. If $X$ is a noetherian separated scheme and $\rmD(X)$ is the derived category of quasi-coherent sheaves on $X$, one can use the action of $\rmD(X)$ on $\rmS(X)$ in a similar fashion to obtain the notions of cosupport and costratification.
\begin{thm*}[\ref{thm:hyper-costrat}, \ref{thm:costrat-schemes}]
Let $R$ be a locally hypersurface ring. Then $\rmS(R)$ is costratified. Let $X$ be a noetherian separated scheme with hypersurface singularities. Then $\rmS(X)$ is costratified.
\end{thm*}

The paper is organized as follows: In~\Cref{sec:singularity}, we recall some basic facts about the tensor-triangular geometry of the derived category of a ring and its action on the singularity category. Then we study the relative internal-hom functor on the singularity category. In~\Cref{sec:gorenstein}, we present a few notions concerning Gorenstein rings and Gorenstein-injective modules, with the key point being that, over a Gorenstein ring, the stable category of Gorenstein-injective modules is equivalent to the singularity category. In~\Cref{sec:endofinite}, using the concept of endofiniteness, we prove that the images of the residue fields under the stabilization functor are pure-injective and in~\Cref{sec:cogen} we obtain cogenerators for certain subcategories of the singularity category. Finally, in~\Cref{sec:main}, we prove~\Cref{thm:hyper-costrat} and in~\Cref{sec:schemes}, we prove~\Cref{thm:costrat-schemes}.
\begin{Ack}
I thank Greg Stevenson for all the helpful discussions, Jan \v{S}t'ov\'{\i}\v{c}ek for bringing to my attention a useful reference and Paul Balmer and Josh Pollitz for their interest and suggestions.
\end{Ack}
\section{Singularity categories}
\label{sec:singularity}%
\begin{conv*}
Throughout, $R$ will always denote a commutative noetherian ring.
\end{conv*}

Let $f$ be an element of $R$. The \emph{stable Koszul complex} associated with $f$ is $K_\infty(f)=R\to R_f$, where $R$ sits in degree $0$ and $R_f$, the localization of $R$ at~$f$, sits in degree $1$ and the map $R\to R_f$ is the localization homomorphism. Let $I=(f_1,\ldots,f_n)$ be an ideal of $R$ and set $K_\infty(I)\coloneqq K_\infty(f_1)\ot_R \cdots \ot_R K_\infty(f_n)$. Up to quasi-isomorphism, $K_\infty(I)$ does not depend on the choice of generators for the ideal $I$. Since $K_\infty(I)$ is a bounded complex of flat $R$-modules, $K_\infty(I)$ is $\rmK$-\emph{flat}, which means that tensoring with $K_\infty(I)$ preserves quasi-isomorphisms.
\begin{lem}\label{lem:koszul-quotient}%
Let $I,J$ be ideals of $R$ with $J\subseteq I$. Then $K_\infty(I/J)\cong R/J\ot_R K_\infty(I)$.
\end{lem}

\begin{proof}
Let $I=(f_1,\ldots,f_n)$. Then $I/J=(\ol{f_1},\ldots,\ol{f_n})$. Since
\[
K_\infty(\ol{f_i})=(R/J\to (R/J)_{\ol{f_i}})
\cong (R/J \ot_R (R \to R_{f_i}))
=(R/J \ot_R K_\infty(f_i)),
\]
we have $K_\infty(I/J)=\bigotimes_{i=1}^n K_\infty(\ol{f_i}) \cong (\bigotimes_{i=1}^n R/J) \ot_R K_\infty(I)$. Since $R/J\ot_R R/J\cong (R/J)/(JR/J)\cong R/J$, we conclude that $K_\infty(I/J) \cong R/J \ot_R K_\infty(I)$.
\end{proof}

\begin{rec}\label{rec:derived}%
The derived category $\rmD(R)$ is a rigidly-compactly generated tensor triangulated category, in the sense of~\cite{BalmerFavi11}, whose subcategory of compact objects is $\rmD^{\mathrm{perf}}(R)$ the subcategory of perfect complexes, i.e., bounded complexes of finitely generated projective $R$-modules up to quasi-isomorphism. There is a homeomorphism $\Spc(\rmD^{\mathrm{perf}}(R))\cong \Spec R$ between the Balmer spectrum of $\rmD(R)$ and $\Spec R$~\cite{Neeman92}. Treating this map as an identification, if $\mfp \in \Spec R$, then the associated Balmer--Favi idempotent is $g_\mfp=K_\infty(\mfp)\ot_R R_\mfp\in \rmD(R)$, where $R_\mfp$ is the localization of $R$ at $\mfp$. Note that $g_\mfp$ is $\rmK$-flat. The Balmer--Favi support of an object $X\in \rmD(R)$ is $\Supp(X)=\{\mfp \in \Spec R\mid g_\mfp \ot_R X\neq 0\}$.
\end{rec}

\begin{rec}\label{rec:singularity}%
The \emph{singularity category} of $R$ is $\rmS(R)\coloneqq \rmK_{\mathrm{ac}}(\Inj R)$ the homotopy category of acyclic complexes of injective $R$-modules. By~\cite{Krause05}, $\rmS(R)$ is a compactly generated triangulated category and there is a recollement
\[
\begin{tikzcd}
\rmS(R) \rar["I"description] \rar[phantom,shift left=2ex,"\perp"] \rar[phantom,shift right=2ex,"\perp"]& \rmK(\Inj R) \rar["Q"description] \rar[phantom,shift left=2ex,"\perp"] \rar[phantom,shift right=2ex,"\perp"] \lar["I_\gl"',shift right=3.2ex] \lar["I_\gr",shift left=3.2ex]& \rmD(R), \lar["Q_\gl"',shift right=3.2ex] \lar["Q_\gr",shift left=3.2ex]
\end{tikzcd}
\]
where $I$ is the inclusion and $Q$ is the composite $\rmK(\Inj R)\hookrightarrow \rmK(R)\twoheadrightarrow \rmD(R)$. The functor $I_\gl Q_\rho\colon \rmD(R)\to \rmS(R)$, called the \emph{stabilization functor}, induces an equivalence of triangulated categories between the idempotent completion of $\rmD_{\mathrm{Sg}}(R)=\rmD^\rmb(\smod R)/\rmD^{\mathrm{perf}}(R)$ and the subcategory of compact objects of $\rmS(R)$. When there are multiple rings involved, we will use the notation $I_\gl Q_\rho^R$.
\end{rec}

\begin{rec}\label{rec:regular}%
The ring $R$ is called a \emph{regular} ring if $R_\mfp$ is a \emph{regular local} ring, i.e., $\gldim R_\mfp<\infty$, for all $\mfp \in \Spec R$ (this is just one of the many equivalent definitions of a regular local ring). It is a fact that $R$ is regular if and only if $\rmS(R)=0$. The \emph{singular locus} of $R$ is $\Sing R=\{ \mfp\in \Spec R \mid \gldim R_\mfp=\infty\}$, i.e., $\Sing R$ consists of those prime ideals $\mfp$ such that $R_\mfp$ is not a regular local ring. Clearly, $R$ is a regular ring if and only if $\Sing R=\varnothing$.

A commutative noetherian local ring $S$ is called a \emph{hypersurface} ring if the completion of $S$ at its unique maximal ideal is isomorphic to the quotient of a regular ring by a regular element. It holds that a hypersurface ring is Gorenstein. We say that $R$ is a \emph{locally hypersurface} ring if $R_\mfp$ is a hypersurface ring, for all $\mfp\in \Spec R$. It holds that a locally hypersurface ring is Gorenstein; see~\Cref{defn:gorenstein}.
\end{rec}

\begin{rec}\label{rec:action}%
By~\cite{Stevenson14}, there is an \emph{action} $-\ast -\colon \rmD(R) \times \rmS(R)\to \rmS(R)$, i.e., a coproduct-preserving triangulated bifunctor such that $X\ast (Y \ast A) \cong (X\ot^\rmL_R Y) \ast A$ and $R\ast A\cong A,\, \forall X,Y\in \rmD(R),\, \forall A\in \rmS(R)$ defined as follows: If $X\in \rmD(R)$ and $A\in \rmS(R)$, then $X\ast A=\wt{X}\ot_R A$, where $\wt{X}$ is a $\rmK$-\emph{flat resolution} of $X$, i.e., $\wt{X}$ is a $\rmK$-flat complex quasi-isomorphic to $X$.
\end{rec}

\begin{rem}\label{rem:kflat-resolutions}%
The action of $\rmD(R)$ on $\rmS(R)$ does not depend on the choice of $\rmK$-flat resolutions, in the sense that different $\rmK$-flat resolutions of the same object yield naturally isomorphic functors. Moreover, any complex of $R$-modules admits a $\rmK$-flat resolution that consists of flat $R$-modules (these are sometimes called semi-flat complexes); see~\cite[Corollary 3.22]{Murfet07} for a more general version of this result concerning schemes. Hence, when we consider the action of $\rmD(R)$ on $\rmS(R)$, we can assume that all $\rmK$-flat resolutions are semi-flat.
\end{rem}
If $X\in \rmD(R)$, then the functor $X\ast -\colon \rmS(R) \to \rmS(R)$ is a coproduct-preserving triangulated functor. Since $\rmS(R)$ is compactly generated, Brown representability implies that $X\ast -$ has a right adjoint $[X,-]$; see~\cite[Theorem 4.1]{Neeman96}. As in~\cite{Verasdanis23a}, we call $[X,-]$ the \emph{relative internal-hom}. Note that $[R,-]\cong \Id_{\rmS(R)}$, since $[R,-]$ is the right adjoint of $R\ast -\cong \Id_{\rmS(R)}$.

Let $\wt{X}$ be a $\rmK$-flat resolution of $X$. We have an adjunction
\begin{equation}\label{eq:adj-on-kr}%
\begin{tikzcd}[column sep=5em]
\rmK(R) \rar[shift left=1.2ex,"\wt{X}\ot_R {-}"] \rar[phantom,"\perp"] & \rmK(R) \lar[shift left=1.2ex,"\Hom_R(\wt{X}{,}{-})"]
\end{tikzcd}
\end{equation}
that restricts to an adjunction on $\rmS(R)$. We will show this next, using a result of Emmanouil~\cite{Emmanouil23} which was based on work of \v{S}t'ov\'{\i}\v{c}ek~\cite{Stovicek14}.
\begin{prop}\label{prop:int-hom well-defined}%
Let $F$ be a $\rmK$-flat complex of flat $R$-modules and let $A$ be an acyclic complex of injective $R$-modules. Then $\Hom_R(F,A)$ is an acyclic complex of injective $R$-modules. In particular, the adjunction~\eqref{eq:adj-on-kr} restricts to an adjunction on $\rmS(R)$.
\end{prop}

\begin{proof}
Since $A$ is an acyclic complex that consists of injective $R$-modules, hence of pure-injective $R$-modules, it holds that $\gS^n A$ is an acyclic complex of pure-injective $R$-modules, for all $n\in\ \bbZ$. According to~\cite[Proposition 3.1]{Emmanouil23}, it follows that $\Hom_{\rmK(R)}(F,\gS^n A)=0$. Thus, $\rmH^n(\Hom_R(F,A))=\Hom_{\rmK(R)}(F,\gS^n A)=0,\ \forall n\in \bbZ$. In other words, $\Hom_R(F,A)$ is acyclic. Each term of the complex $\Hom_R(F,A)$ is a product of $R$-modules of the form $\Hom_R(M,N)$, where $M$ is a flat $R$-module and $N$ is an injective $R$-module. So, the functor $\Hom_R(-,\Hom_R(M,N))$, which is naturally isomorphic to $\Hom_R(M\ot_R -,N)$, is exact. Equivalently, $\Hom_R(M,N)$ is an injective $R$-module. Since injective $R$-modules are closed under products, $\Hom_R(F,A)$ consists of injective $R$-modules.

By what we just proved, we conclude that the restrictions of the functors involved in~\eqref{eq:adj-on-kr} on $\rmS(R)$ take values in $\rmS(R)$. Hence, we obtain the adjunction
\[
\begin{tikzcd}[column sep=5em]
\rmS(R) \rar[shift left=1.2ex,"\wt{X}\ot_R {-}"] \rar[phantom,"\perp"] & \rmS(R). \lar[shift left=1.2ex,"\Hom_R(\wt{X}{,}{-})"]
\end{tikzcd}\qedhere
\]
\end{proof}

\begin{cor}\label{cor:rel int-hom is hom}%
Let $X\in \rmD(R)$ and let $\wt{X}$ be a $\rmK$-flat resolution of $X$ that consists of flat $R$-modules. Then $[X,-] = \Hom_R(\wt{X},-)\colon \rmS(R)\to \rmS(R)$.
\end{cor}

\begin{proof}
The claim follows immediately from~\Cref{prop:int-hom well-defined} due to the fact that $\Hom_R(\wt{X},-)\colon \rmS(R)\to \rmS(R)$ is right adjoint to $X\ast -=\wt{X}\ot_R -\colon \rmS(R)\to \rmS(R)$.
\end{proof}

Every colocalizing subcategory of $\rmS(R)$ is a $\shom$-\emph{submodule} over $\rmD(R)$, in the sense of the following lemma:
\begin{lem}\label{lem:all-are-hom-submodules}%
Let $\scC$ be a colocalizing subcategory of $\rmS(R)$. Then it holds that $[X,A]\in \scC,\, \forall X\in \rmD(R),\, \forall A\in \scC$.
\end{lem}

\begin{proof}
Let $A\in \scC$ and set $\scX=\{X\in \rmD(R) \mid [X,A]\in \scC\}$. It is straightforward to verify that $\scX$ is a localizing subcategory of $\rmD(R)$, taking into account~\Cref{cor:rel int-hom is hom} and~\cite[Remark 6.1]{Verasdanis23a}. Moreover, $[R,A]\cong A\in \scC$ implies $R\in \scX$. Since $\rmD(R)=\loc(R)$, it follows that $\scX=\rmD(R)$, which proves the statement.
\end{proof}

\section{Gorenstein rings}
\label{sec:gorenstein}
In this section, we recall some facts about Gorenstein rings, Gorenstein-injective and Gorenstein-projective modules and the stable category of Gorenstein-injective modules.
\begin{defn}\label{defn:gorenstein}%
Let $R$ be a commutative noetherian ring and let $M$ be an $R$-module.
\begin{enumerate}[\rm(a)]
\item
The ring $R$ is called \emph{Gorenstein} if $R$ has finite injective dimension as an $R$-module.
\item
The $R$-module $M$ is called \emph{Gorenstein-injective} if there exists an acyclic complex $C$ that consists of injective $R$-modules such that $\Hom_R(I,C)$ is acyclic, for all injective $R$-modules $I$ and $M=\rmZ^0 C$ the kernel of the zeroth differential of $C$. Such a complex $C$ is called a \emph{complete injective resolution} of $M$.
\item
The $R$-module $M$ is called \emph{Gorenstein-projective} if there exists an acyclic complex $C$ that consists of projective $R$-modules such that $\Hom_R(C,P)$ is acyclic, for all projective $R$-modules $P$ and $M=\rmZ^0 C$ the kernel of the zeroth differential of $C$. Such a complex is called a \emph{complete projective resolution} of $M$.
\item
The \emph{Gorenstein-injective envelope} of $M$ is (if it exists) a Gorenstein-injective $R$-module $G_R(M)$ together with a morphism $f\colon M\to G_R(M)$ such that the following two conditions hold: First, for all Gorenstein-injective $R$-modules $G$ and morphisms $g\colon M\to G$, there exists a morphism $h\colon G_R(M)\to G$ such that $g=h\circ f$. Second, if $h\colon G_R(M)\to G_R(M)$ is a morphism such that $h\circ f=f$, then $h$ is an isomorphism.
\item
The \emph{Gorenstein-projective cover} of $M$ is (if it exists) a Gorenstein-projective $R$-module $G^R(M)$ together with a morphism $f\colon G^R(M)\to M$ such that the following two conditions hold: First, for all Gorenstein-projective $R$-modules $G$ and morphisms $g\colon G\to M$, there exists a morphism $h\colon G\to G^R(M)$ such that $g=f\circ h$. Second, if $h\colon G^R(M)\to G^R(M)$ is a morphism such that $f=f\circ h$, then $h$ is an isomorphism.
\end{enumerate}
\end{defn}

\begin{prop}[{\cite[Theorem 11.3.2, Theorem 11.6.9]{EnochsJenda00}}]\label{prop:envelopes-covers}%
If $R$ is a Gorenstein ring, then any $R$-module admits a Gorenstein-injective envelope. If, moreover, $R$ is local, then any finitely generated $R$-module admits a finitely generated Gorenstein-projective cover.
\end{prop}

\begin{rec}\label{rec:ginj-sing-equiv}%
The category $\ginj R$ of Gorenstein-injective $R$-modules is an exact subcategory of $\Mod R$ with exact sequences those short exact sequences of Gorenstein-injective $R$-modules. In fact, $\ginj R$ is a Frobenius exact category, i.e., $\ginj R$ has enough projectives and enough injectives and its projective and injective objects coincide: they are precisely the injective $R$-modules. So, $\sginj R$ the stable category of Gorenstein-injective $R$-modules is a triangulated category. According to~\cite[Proposition 7.13]{Krause05}, if $R$ is a Gorenstein ring, there is an equivalence of triangulated categories $\rmS(R) \xr{\simeq} \sginj R$ given by mapping $A\in \rmS(R)$ to $\rmZ^0 A$ with inverse given by sending a Gorenstein-injective $R$-module $M$ to a complete injective resolution $C(M)$. Furthermore, by~\cite[Corollary 4.8]{Stevenson14}, the functors $G_R,\rmZ^0 I_\gl Q_\rho\colon \Mod R \to \sginj R$ are naturally isomorphic.
\end{rec}

\section{Endofiniteness, pure-injectivity and residue fields}
\label{sec:endofinite}%
Let $\scT$ be a compactly generated triangulated category. The category $\Mod(\scT^\rmc)$ of additive functors $\{\scT^\rmc\}^\rmop \to \mathrm{Ab}$ is a (Grothendieck) abelian category and there is a functor $\wh{(-)}\colon \scT \to \Mod(\scT^\rmc)$ that sends $X\in \scT$ to $\wh{X}=\Hom_\scT(-,X)\rvert_{\scT^\rmc}\in \Mod(\scT^\rmc)$ called the \emph{restricted Yoneda functor}. A morphism $f\colon X\to Y$ in $\scT$ is called a \emph{pure monomorphism} if $\wh{f}\colon \wh{X} \to \wh{Y}$ is a monomorphism in $\Mod(\scT^\rmc)$. An object $X \in \scT$ is called \emph{pure-injective} if every pure-monomorphism $f\colon X\to Y$ splits (meaning that there exists a morphism $g\colon Y\to X$ such that $g\circ f=1_X$). By~\cite[Theorem 1.8]{Krause00}, an object $X\in \scT$ is pure-injective if and only if $\wh{X}$ is an injective object of $\Mod(\scT^\rmc)$. (In fact, an object $E\in \Mod(\scT^\rmc)$ is injective if and only if there exists a pure-injective object $X\in \scT$ such that $E\cong \wh{X}$; see~\cite[Corollary 1.9]{Krause00}.)

\begin{rem}\label{rem:right-adjoints-preserve-pinj}%
Let $\scT$ and $\scU$ be compactly generated triangulated categories and let $F\colon \scT \to \scU$ be a coproduct-preserving triangulated functor. Then $F$ has a right adjoint $G\colon \scU\to \scT$; see~\cite[Theorem 4.1]{Neeman96}. If $G$ preserves coproducts, then $G$ preserves pure-injective objects: There is an induced adjunction
\[
\begin{tikzcd}[column sep=5em]
\Mod(\scT^\rmc) \rar["\ol{F}",shift left=1.5ex] \rar[phantom,"\perp"] & \Mod(\scU^\rmc), \lar["\ol{G}",shift left=1.5ex]
\end{tikzcd}
\]
where $\ol{F}$ is an exact functor and $\ol{F}$ and $\ol{G}$ commute with the restricted Yoneda functors. Since $\ol{G}$ is the right adjoint of an exact functor, $\ol{G}$ preserves injective objects. By the characterization of pure-injective objects as precisely those objects whose image under the restricted Yoneda functor is injective, it follows that $\ol{G}$ preserves pure-injective objects; see~\cite[Proposition 2.6]{Krause00}.
\end{rem}

\begin{rem}\label{rem:res-preserves-pinj}%
Let $\mfp$ be a prime ideal of $R$. Since $R_\mfp\ot_R - \colon \Mod R\to \Mod R_\mfp$ is an exact functor with right adjoint $\res \colon \Mod R_\mfp \to \Mod R$, it follows that $\res$ preserves injective modules. Moreover, since $R$ is noetherian, localization preserves injective modules, so $R_\mfp\ot_R -$ preserves injective modules. Consequently, we have an adjunction
\[
\begin{tikzcd}[column sep=5em]
\rmS(R) \rar["R_\mfp\ot_R -",shift left=1.5ex] \rar[phantom,"\perp"] & \rmS(R_\mfp). \lar["\res",shift left=1.5ex]
\end{tikzcd}
\]
Since the categories $\rmS(R)$ and $\rmS(R_\mfp)$ are compactly generated and $\res$ preserves coproducts, by~\Cref{rem:right-adjoints-preserve-pinj}, it follows that $\res\colon \rmS(R_\mfp)\to \rmS(R)$ preserves pure-injective objects.
\end{rem}

An object $X$ of a triangulated category $\scT$ is called \emph{endofinite} if, for all compact objects $C$ of $\scT$, it holds that $\Hom_\scT(C,X)$ is a finite length module over $\End_\scT(X)$. By~\cite[Theorem 1.2]{Krause99} (see also~\cite[Proposition 3.3]{KrauseReichenbach00} and~\cite[Proposition 5.6]{Krause23}) endofinite objects are pure-injective.

\begin{prop}\label{prop:residue field endofinite}%
Let $R=(R,\mfm,k)$ be a local Gorenstein ring. Then $I_\gl Q_\rho (k)$ is an endofinite (hence a pure-injective) object of $\rmS(R)$.
\end{prop}

\begin{proof}
Since $k=R/\mfm$ is a finitely generated $R$-module, it follows that $k$ is an object of $\rmD_{\mathrm{Sg}}(R)$. As we have already discussed in~\Cref{rec:singularity}, the subcategory of compact objects of $\rmS(R)$ is equivalent to the closure under summands of the image of $\rmD_{\mathrm{Sg}}(R)$ under $I_\gl Q_\rho$. Further, by~\cite[Lemma 1.11]{Orlov04}, every object of $\rmD_{\mathrm{Sg}}(R)$ is of the form $\gS^i M$, where $M$ is a finitely generated $R$-module. Consequently, it suffices to show that $\Hom_{\rmD_{\mathrm{Sg}}(R)}(\gS^i M,k)$ is a module of finite length over $\End_{\rmD_{\mathrm{Sg}}(R)}(k)\cong \End_R(k)\cong k$, i.e., a finite dimensional $k$-vector space. Since $R$ is a local Gorenstein ring and $M$ is a finitely generated $R$-module, it follows that $M$ has a finitely generated Gorenstein-projective cover $G^R(M)$; see~\Cref{prop:envelopes-covers}. Let $c(M)$ be a complete projective resolution of $G^R(M)$ that consists of finitely generated projective $R$-modules. Now we compute:
\begin{align*}
\Hom_{\rmD_{\mathrm{Sg}}(R)}(\gS^iM,k)&\cong \ul{\Hom}_R(M,\gS^{-i}G_R(k))\\
&\cong \wh{\Ext}_R^{-i}(M,k)\\
&=\rmH^{-i}(\Hom_R(c(M),k))
\end{align*}
with the second isomorphism by~\cite[Proposition 7.7]{Krause05} and the subsequent equality by definition. Since $c(M)$ is a complex of finitely generated $R$-modules, it follows that $\Hom_R(c(M),k)$ is a complex of finite dimensional $k$-vector spaces. Hence, $\rmH^{-i}(\Hom_R(c(M),k))$ is a finite dimensional $k$-vector space. Consequently, $\Hom_{\rmD_\mathrm{Sg}(R)}(\gS^i M,k)$ is a finite dimensional $k$-vector space. In conclusion, $k$ is an endofinite object of $\rmD_{\mathrm{Sg}}(R)$ and so $I_\gl Q_\rho (k)$ is an endofinite (hence a pure-injective) object of $\rmS(R)$.
\end{proof}

\begin{prop}\label{prop:residue-pinj}%
Let $R$ be a Gorenstein ring. Then $I_\gl Q_\rho (k(\mfp))$ is a pure-injective object of $\rmS(R)$.
\end{prop}

\begin{proof}
Since $R$ is a Gorenstein ring, it follows that $R_\mfp$ is a local Gorenstein ring. Hence, by~\Cref{prop:residue field endofinite}, $I_\gl Q_\rho^{R_\mfp}(k(\mfp))$ is a pure-injective object of $\rmS(R_\mfp)$. We have $I_\gl Q_\rho^R(k(\mfp))=I_\gl Q_\rho^R(\res k(\mfp))=\res I_\gl Q_\rho^{R_\mfp}(k(\mfp))$. According to~\Cref{rem:res-preserves-pinj}, $\res \colon \rmS(R_\mfp) \to \rmS(R)$ preserves pure-injective objects. Hence, $I_\gl Q_\rho^R(k(\mfp))$ is pure-injective.
\end{proof}

\section{Cogeneration}
\label{sec:cogen}%
Our goal here is to prove that if $R=(R,\mfm,k)$ is a hypersurface ring, then the image of the functor $[g_\mfm,-]\colon \rmS(R)\to \rmS(R)$ is cogenerated by $I_\gl Q_\rho (k)$. The key results we will need are~\Cref{prop:residue-pinj} and the following:
\begin{lem}\label{lem:hom-koszul}%
Let $R=(R,\mfm,k)$ be a local Gorenstein ring. Then, in $\rmS(R)$,
\begin{equation}\label{eq:hom-gm}
\Hom_R(K_\infty(\mfm),I_\gl Q_\rho(k(\mfp)))\cong
\begin{cases}
I_\gl Q_\rho (k), & \mfp=\mfm,
\\
0, & \mfp \neq \mfm.
\end{cases}
\end{equation}
\end{lem}

\begin{proof}
We will first prove~\eqref{eq:hom-gm} for the case $\mfp=\mfm$, by induction on the Krull dimension of $R$. In the case $\dim R=0$, we have $\Spec R=\{\mfm\}$. Thus, $\mfm$ consists of nilpotent elements. Let $\mfm=(f_1,\ldots,f_n)$. Then each $f_i$ is nilpotent. Therefore, $R_{f_i}=0$ and so $K_\infty(\mfm)=\bigotimes_{i=1}^n K_\infty(f_i)=\bigotimes_{i=1}^n R\cong R$. We have $\Hom_R(K_\infty(\mfm),I_\gl Q_\rho(k))\cong \Hom_R(R,I_\gl Q_\rho (k))\cong I_\gl Q_\rho (k)$.

Now let $d>0$ and assume that~\eqref{eq:hom-gm} (for $\mfp=\mfm$) holds for all local Gorenstein rings of dimension strictly less than $d$ and suppose that $\dim R=d$. Let $x\in \mfm$ be a regular element (such an element exists by the prime avoidance lemma and our assumption that $\dim R=d>0$). Then $R/(x)$ is a local Gorenstein ring with residue field $(R/(x))/(\mfm/(x))\cong R/\mfm=k$ and $\dim R/(x)=d-1$. By~\Cref{lem:koszul-quotient}, $K_\infty(\mfm/(x))\cong R/(x)\ot_R K_\infty(\mfm)$. We have
\begin{align*}
\Hom_R(K_\infty(\mfm),I_\gl Q_\rho^R (k))&\cong\Hom_R(K_\infty(\mfm),\res I_\gl Q_\rho^{R/(x)} (k))\\
&\cong\res \Hom_{R/(x)}(R/(x)\ot_R K_\infty(\mfm),I_\gl Q_\rho^{R/(x)}(k))\\
&\cong\res \Hom_{R/(x)}(K_\infty(\mfm/(x)),I_\gl Q_\rho^{R/(x)}(k))\\
&\cong\res I_\gl Q_\rho^{R/(x)}(k) \\
&\cong I_\gl Q_\rho^R(k).
\end{align*}
The first and last isomorphisms hold because the singularity category of a Gorenstein ring is equivalent to the stable category of Gorenstein-injective modules and $\rmZ^0 I_\gl Q_\rho=G(-)$ on modules (see~\Cref{rec:ginj-sing-equiv}) and by~\cite[Remark 6.11]{Stevenson14}, taking Gorenstein-injective envelopes commutes with restriction. The second isomorphism follows from the internal-hom version of the adjunction
\[
\begin{tikzcd}[column sep=5em]
\rmK(R) \rar["R/(x)\ot_R -",shift left=1.5ex] \rar[phantom,"\perp"] & \rmK(R/(x)), \lar["\res",shift left=1.5ex]
\end{tikzcd}
\]
which asserts that the functors
\[
\res \Hom_{R/(x)}(R/(x)\ot_R -,-),\, \Hom_R(-,\res(-))\colon \rmK(R)^\rmop \times \rmK(R/(x))\to \rmK(R)
\]
are naturally isomorphic. The fourth isomorphism holds by the inductive hypothesis. This completes the proof of~\eqref{eq:hom-gm} for the case $\mfp=\mfm$.

Let $\mfp\neq \mfm$ be a prime ideal of $R$. Then necessarily $\mfp\subsetneq \mfm$ since $\mfm$ contains all prime ideals. We have
\begin{align*}
\Hom_R(K_\infty(\mfm),I_\gl Q_\rho^R(k(\mfp)))&\cong \Hom_R(K_\infty(\mfm),\res I_\gl Q_\rho^{R_\mfp}(k(\mfp)))\\
&\cong\res \Hom_{R_\mfp}(R_\mfp\ot_R K_\infty(\mfm),I_\gl Q_\rho^{R_\mfp}(k(\mfp)))\\
&=0.
\end{align*}
The first two isomorphisms are justified in the same way as in the calculation in the previous paragraph, replacing $R/(x)$ with $R_\mfp$. Since $\mfp \subsetneq \mfm$, it follows that $\loc(K_\infty(\mfm))=\scM \subsetneq \scP=\Ker (R_\mfp \ot_R -)$, where $\scP$ and $\scM$ are the smashing subcategories of $\rmD(R)$ corresponding to $\mfp$ and~$\mfm$, respectively. Hence, $K_\infty(\mfm)\ot_R R_\mfp=0$. This explains the last equality, completing the proof.
\end{proof}

The following lemma is well-known and easy to prove. We present it for the convenience of the reader.
\begin{lem}\label{lem:image-of-coloc}%
Let $G\colon \scT \to \scU$ be a product-preserving triangulated functor between triangulated categories with products and let $\scX$ be a collection of objects of $\scT$. Then $G\coloc(\scX)\subseteq \coloc(G\scX)$.
\end{lem}

\begin{proof}
It is straightforward to verify that $\scL=\{X\in \scT\mid GX\in \coloc(G\scX)\}$ is a colocalizing subcategory of $\scT$ that contains $\scX$. It follows that $\coloc(\scX)\subseteq \scL$, which proves the statement.
\end{proof}

\begin{rec}\label{rec:perfect-cogen}%
Let $\scT$ be a compactly generated triangulated category. If $\scX$ is a cogenerating set of objects of $\scT$ (in the sense that $^\perp \scX=0$) that consists of pure-injective objects, then $\scT=\coloc(\scX)$, which means that the smallest colocalizing subcategory of $\scT$ that contains $\scX$ is $\scT$. This holds because the pure-injective objects of $\scT$ form a perfect cogenerating set in the sense of~\cite{Krause02}; see~\cite[Section 3]{Verasdanis23a} for some explanations. Also,~\cite[Section 9]{BarthelCastellanaHeardSanders23} provides a detailed account on the concept of perfect (co)generation.
\end{rec}

\begin{prop}\label{prop:sing-cogen}%
Let $R$ be a locally hypersurface ring. Then
\[
\rmS(R)=\coloc(I_\gl Q_\rho(k(\mfp)) \mid \mfp \in \Sing R).
\]
\end{prop}

\begin{proof}
Let $\scL={^\perp \{\gS^n I_\gl Q_\gr (k(\mfp)) \mid n\in \bbZ,\, \mfp \in \Sing R\}}$. Then $\scL$ is a localizing subcategory of $\rmS(R)$ that does not contain any $I_\gl Q_\gr (k(\mfp))$, for $\mfp \in \Sing R$. Indeed, if $I_\gl Q_\rho (k(\mfp))\in \scL$, then $\Hom_{\rmS(R)}(I_\gl Q_\rho k(\mfp),I_\gl Q_\rho k(\mfp))=0$ and this implies that $I_\gl Q_\rho (k(\mfp))=0$, which is false when $\mfp \in \Sing R$. Consequently, by~\cite[Theorem 6.13]{Stevenson14}, we have $\scL=0$, i.e., $\{\gS^n I_\gl Q_\gr (k(\mfp)) \mid n\in \bbZ,\, \mfp\in \Sing R\}$ is a cogenerating set for $\rmS(R)$. Since, by~\Cref{prop:residue-pinj}, the objects $I_\gl Q_\gr (k(\mfp))$ are pure-injective, it follows that $\rmS(R)=\coloc(I_\gl Q_\gr (k(\mfp)) \mid \mfp \in \Sing R)$; see~\Cref{rec:perfect-cogen}.
\end{proof}

\begin{prop}\label{prop:stalk-maximal}%
Let $R=(R,\mfm,k)$ be a hypersurface ring. Then $[g_\mfm,\rmS(R)]=\coloc(I_\gl Q_\rho (k))$.
\end{prop}

\begin{proof}
Since $[g_\mfm,-]=\Hom_R(g_\mfm,-)\colon \rmS(R)\to \rmS(R)$ is a product-preserving triangulated functor, it follows by~\Cref{lem:image-of-coloc} and~\Cref{prop:sing-cogen} that $[g_\mfm,\rmS(R)]=\coloc([g_\mfm,I_\gl Q_\gr (k(\mfp))] \mid \mfp \in \Sing R)=\coloc(I_\gl Q_\rho (k))$, with the last equality due to~\Cref{lem:hom-koszul}.
\end{proof}

\section{Locally hypersurface rings}
\label{sec:main}%
In this section, we prove~\Cref{thm:hyper-costrat}, which classifies the colocalizing subcategories of the singularity category $\rmS(R)$ of a locally hypersurface ring $R$ in terms of the singular locus $\Sing R$.
\subsection{Cosupport and costratification}
Let $A$ be an object of $\rmS(R)$. The \emph{cosupport} of $A$ is $\Cosupp(A)=\{\mfp \in \Spec R \mid \Hom_R(g_\mfp,A)\neq 0\}$.
\begin{lem}\label{lem:cosupport-properties}%
The assignment $\Cosupp\colon \Ob \rmS(R) \to \scP(\Spec R)$ satisfies the following properties:
\begin{enumerate}[\rm(a)]
\item
$\Cosupp(0)=\varnothing$.
\item
$\Cosupp(\prod A_i)=\bigcup \Cosupp(A_i)$.
\item
$\Cosupp(\gS A)=\Cosupp(A)$.
\item
$\Cosupp(A)\subseteq \Cosupp(B) \cup \Cosupp(C)$, for all triangles $A\to B\to C$ of $\rmS(R)$.
\item
$\Cosupp([X,A])\subseteq \Supp(X) \cap \Cosupp(A)$.
\end{enumerate}
\end{lem}

\begin{proof}
See~\cite[Lemma 3.12]{Verasdanis23a}.
\end{proof}

\begin{rem}
Recall that by~\Cref{lem:all-are-hom-submodules}, all colocalizing subcategories of $\rmS(R)$ are $\shom$-submodules, in the sense of~\cite{Verasdanis23a}, and so all results concerning colocalizing $\shom$-submodules apply to all colocalizing subcategories of $\rmS(R)$.
\end{rem}

We denote by $I$ the product of the Brown--Comenetz duals of the compact objects of $\rmS(R)$. Then $I$ is a pure-injective cogenerator of $\rmS(R)$; see~\Cref{rec:perfect-cogen}.

We define the maps
\[
\begin{tikzcd}
\Coloc(\rmS(R)) \rar["\sigma",shift left] & \scP(\Spec R){,} \lar["\tau",shift left]
\end{tikzcd}
\]
where $\sigma(\scC)=\bigcup_{A\in \scC}\Cosupp(A)$ and $\tau(W)=\{A\in \rmS(R) \mid \Cosupp(A)\subseteq W\}$. The maps $\sigma$ and $\tau$ are inclusion-preserving. By~\cite[Section 3.B]{Verasdanis23a} and~\cite[Proposition 5.7]{Stevenson14}, it follows that $\sigma(\rmS(R))=\Cosupp(I)=\Sing R$. Hence, $\sigma(\scC)\subseteq \sigma(\rmS(R))=\Sing R$. This shows that $\sigma\colon \Coloc(\rmS(R))\to \scP(\Sing R)$ is well-defined. From now on, we will consider the codomain of $\sigma$ and the domain of $\tau$ to be $\scP(\Sing R)$. It holds that $\sigma \circ \tau = \Id$; see~\cite[Lemma 3.21]{Verasdanis23a}.

\begin{defn}[{\cite[Definition 3.13]{Verasdanis23a}}]
\label{defn:costrat}%
If $\sigma$ and $\tau$ are bijections inverse to each other (which holds when $\tau \circ \sigma=\Id$) then we say that $\rmS(R)$ is \emph{costratified}.
\end{defn}

Our goal is to prove that if $R$ is a locally hypersurface ring, then $\rmS(R)$ is costratified. For that we will use an equivalent characterization of costratification in terms of certain colocalizing subcategories.
\begin{defn}[{\cite[Definition 3.16]{Verasdanis23a}}]
\label{defn:coltg-comin}%
$\phantom{}$
\begin{enumerate}[\rm(a)]
\item
$\rmS(R)$ satisfies the \emph{colocal-to-global principle} if
\[
\coloc(A)=\coloc(\Hom_R(g_\mfp,A)\mid \mfp \in \Sing R),\, \forall A\in \rmS(R).
\]
\item
$\rmS(R)$ satisfies \emph{cominimality} if $\coloc(\Hom_R(g_\mfp,I))$ is a minimal colocalizing subcategory of $\rmS(R),\, \forall \mfp \in \Sing R$.
\end{enumerate}
\end{defn}

According to~\cite[Theorem 3.22]{Verasdanis23a}, $\rmS(R)$ is costratified if and only if $\rmS(R)$ satisfies the colocal-to-global principle and cominimality. This will allow us to prove that if $R$ is a hypersurface ring, then $\rmS(R)$ is costratified. Proving costratification of $\rmS(R)$ when $R$ is a locally hypersurface ring amounts to reducing costratification to certain smashing localizations, as described in~\cite[Theorem 5.4]{Verasdanis23a}.

\subsection{Colocalizing subcategories of $\rmS(R)$, for a locally hypersurface ring $R$}
\label{subsec:main}%
\begin{lem}\label{lem:image-under-hom}%
Let $R=(R,\mfm,k)$ be a local Gorenstein ring and let $x\in R$ be a regular element. Let $G$ be a non-zero object of $[g_\mfm,\rmS(R)]$ and set $\wt{M}=\Hom_R(R/(x),G)$ viewed as a complex of $R/(x)$-modules. Then $\wt{M}\in [g_{\mfm/(x)},\rmS(R/(x))]$ and $\wt{M}\neq 0$.
\end{lem}

\begin{proof}
Since $g_\mfm=K_\infty(\mfm)$ is a left idempotent, it holds that $[g_\mfm,\rmS(R)]=\im [g_\mfm,-]$. Therefore, $\Hom_R(g_\mfm,G)=[g_\mfm,G]\cong G$. We have
\begin{align*}
\res [g_{\mfm/(x)},\wt{M}]&=\res \Hom_{R/(x)}(g_\mfm \ot_R R/(x),\wt{M})\\
&\cong \Hom_R(g_\mfm,\res \wt{M}) \\
&=\Hom_R(g_\mfm,\Hom_R(R/(x),G))\\
& \cong \Hom_R(g_\mfm\ot_R R/(x),G) \\
& \cong \Hom_R(R/(x),\Hom_R(g_\mfm,G)) \\
& \cong \Hom_R(R/(x),G) \\
&=\res \wt{M}.
\end{align*}
Hence, $[g_{\mfm/(x)},\wt{M}]\cong \wt{M}$. This shows that $\wt{M}\in \im[g_{\mfm/(x)},-]=[g_{\mfm/(x)},\rmS(R/(x))]$. By the equivalence between $\rmS(R)$ and $\sginj R$, as described in~\Cref{rec:ginj-sing-equiv}, if $G'$ is the object of $\sginj R$ corresponding to $G$, then $\pd_R \Hom_R(R/(x),G')=\mathrm{id}_R \Hom_R(R/(x),G')=\infty$; see~\cite[Lemma 6.6]{Stevenson14}. Thus, $\Hom_R(R/(x),G')\neq 0$ and so $\Hom_R(R/(x),G)\neq 0$. Since $\Hom_R(R/(x),G)$ is the restriction of $\wt{M}$, we conclude that $\wt{M}\neq 0$.
\end{proof}

A compactly generated triangulated category $\scT$ is called \emph{pure-semisimple} if every object of $\scT$ is pure-injective.

\begin{lem}\label{lem:ps-cg-coloc}%
Let $\scT$ be a pure-semisimple triangulated category such that the only localizing subcategories of $\scT$ are $0$ and $\scT$. Then the only colocalizing subcategories of $\scT$ are $0$ and $\scT$.
\end{lem}

\begin{proof}
Let $X$ be a non-zero object of $\scT$. Then $^\perp \coloc(X)$ is either $0$ or $\scT$. The latter is false, since in that case $X$ would have to be $0$. So, ${^\perp\{\gS^n X \mid n\in \bbZ\}}={^\perp \coloc(X)}=0$. This means that the set of suspensions of $X$ is a cogenerating set of $\scT$. Since $\scT$ is pure-semisimple, $X$ is pure-injective. Consequently, by~\Cref{rec:perfect-cogen}, $\scT=\coloc(X)$. This shows that $\scT$ is cogenerated by any of its non-zero objects. As a result, the only colocalizing subcategories of $\scT$ are $0$ and $\scT$.
\end{proof}

\begin{prop}\label{prop:art-min-hyper}%
Let $R$ be an artinian hypersurface ring with unique maximal ideal $\mfm$. Then $[g_\mfm,\sginj R]$ $($resp.~$[g_\mfm,\rmS(R)]$$)$ is a minimal colocalizing subcategory of $\sginj R$ $($resp.~$\rmS(R)$$)$.
\end{prop}

\begin{proof}
Every $R$-module is Gorenstein-injective, i.e., $\sMod R=\sginj R$ and further, $\sMod R$ is a pure-semisimple compactly generated triangulated category; see the explanations in the proof of~\cite[Lemma 6.8]{Stevenson14}. It follows by~\Cref{lem:ps-cg-coloc} that the only colocalizing subcategories of $\sMod R$ are $0$ and $\sMod R$. Hence, $[g_\mfm,\sginj R]=\sMod R=\sginj R$ is a minimal colocalizing subcategory of $\sginj R$. By the equivalence $\rmS(R)\simeq \sginj R$, it also holds that $[g_\mfm,\rmS(R)]=\rmS(R)$ is a minimal colocalizing subcategory of $\rmS(R)$.
\end{proof}

\begin{prop}\label{prop:min-hyper}%
Let $R=(R,\mfm,k)$ be a hypersurface ring. Then $[g_\mfm,\rmS(R)]$ is a minimal colocalizing subcategory of $\rmS(R)$.
\end{prop}

\begin{proof}
If $\dim R=0$, then $R$ is an artinian hypersurface and the claim holds by~\Cref{prop:art-min-hyper}. Now suppose that $\dim R=d>0$ and that the claim holds for all hypersurface rings of dimension strictly less than $d$. There exists a regular element $x\in R$ such that $R/(x)$ is a hypersurface and $\dim R/(x)=d-1$; see the explanations in the proof of~\cite[Theorem 6.12]{Stevenson14}. Let $G$ be a non-zero object of $[g_\mfm,\rmS(R) ]$ and set $M=\Hom_R(R/(x),G)$ and denote by $\wt{M}$ the complex $M$ viewed as a complex of $R/(x)$-modules. By~\Cref{lem:image-under-hom}, it holds that $\wt{M}\in [g_{\mfm/(x)},\rmS(R/(x))]$. Further, $\res \wt{M}=M\in \coloc(G)$. The last assertion holds because $\coloc(G)$ is a $\hom$-submodule; see~\Cref{lem:all-are-hom-submodules} and~\Cref{rec:ginj-sing-equiv}. By the inductive hypothesis, $[g_{\mfm/(x)},\rmS(R/(x))]$ is minimal and by~\Cref{lem:image-under-hom}, $\wt{M}\neq 0$. Hence, $\coloc(\wt{M})=[g_{\mfm/(x)},\rmS(R/(x))]=\coloc(I_\gl Q_\rho^{R/(x)}(k))$, with the last equality by~\Cref{prop:stalk-maximal}. Since $\res$ is a product-preserving triangulated functor, it follows by~\Cref{lem:image-of-coloc} that $\res \coloc(\wt{M})\subseteq \coloc(\res \wt{M})=\coloc(M)$. The latter is contained in $\coloc(G)$. Consequently, $\res \coloc(I_\gl Q_\rho^{R/(x)}(k))\subseteq \coloc(G)$ and so $I_\gl Q_\rho^R(k)=\res I_\gl Q_\rho^{R/(x)}(k)\in \coloc(G)$. So, $\coloc(I_\gl Q_\rho^R(k))\subseteq \coloc(G)$. We infer that $\coloc(G)=\coloc(I_\gl Q_\rho^R(k))=[g_\mfm,\rmS(R)]$ (with the last equality by~\Cref{prop:stalk-maximal}) and so $[g_\mfm,\rmS(R)]$ is a minimal colocalizing subcategory of $\rmS(R)$.
\end{proof}

\begin{thm}\label{thm:hyper-costrat}%
Let $R$ be a locally hypersurface ring. Then $\rmS(R)$ is costratified.
\end{thm}

\begin{proof}
Since $\rmD(R)$ satisfies the local-to-global principle, $\rmS(R)$ satisfies the colocal-to-global principle; see~\cite[Proposition 3.27]{Verasdanis23a}. Let $\mfp\in \Sing (R)$ and set $\scM_\mfp=\Ker(R_\mfp \ot_R -\colon \rmS(R) \to \rmS(R))$. Then $\scM_\mfp$ is a smashing subcategory of $\rmS(R)$ and $\rmS(R_\mfp)\simeq \rmS(R)/\scM_\mfp$ and ${\scM_\mfp}^\perp=\im(R_\mfp \ot_R -)=\im [R_\mfp,-]$. So, $[g_\mfp,I_{\rmS(R)}]=[K_\infty(\mfp)\ot_R R_\mfp,I_{\rmS(R)}]\cong[R_\mfp,[K_\infty(\mfp),I_{\rmS(R)}]]\in {\scM_\mfp}^\perp$. Since $R_\mfp$ is a hypersurface, by~\Cref{prop:min-hyper}, $\rmS(R_\mfp)$ satisfies cominimality at the unique closed point of $\Spec R_\mfp$. By~\cite[Proposition 5.3]{Verasdanis23a}, $\rmS(R)$ satisfies cominimality at $\mfp$. Hence, $\rmS(R)$ is costratified; see~\cite[Theorem 3.22]{Verasdanis23a}.
\end{proof}

A proper colocalizing subcategory $\scC$ of $\rmS(R)$ is called $\shom$-\emph{prime} if, for all $X\in \rmD(R)$ and $A\in \rmS(R)$, if $[X,A]\in \scC$, then $[X,I_{\rmS(R)}]\in \scC$ or $A\in \scC$. 
\begin{thm}\label{thm:hom-primes}%
Let $R$ be a locally hypersurface ring. Then there is a bijective correspondence between points of $\Sing R$ and $\shom$-prime colocalizing subcategories of $\rmS(R)$. A point $\mfp \in \Sing R$ is associated with $\Ker (\Hom_R(g_\mfp,-)\colon \rmS(R) \to \rmS(R))$, with the latter being equal to $\coloc(\Hom_R(g_\mfq,I_{\rmS(R)}) \mid \mfq \neq \mfp)$.
\end{thm}

\begin{proof}
According to~\Cref{thm:hyper-costrat}, $\rmS(R)$ is costratified. By~\Cref{lem:all-are-hom-submodules}, every colocalizing subcategory of $\rmS(R)$ is a $\shom$-submodule. Further, $\Sing R=\Cosupp(I_{\rmS(R)})$. The claim now follows by applying~\cite[Theorem 4.10]{Verasdanis23a}.
\end{proof}

\begin{ex}\label{ex:dual numbers}%
Let $k$ be a field. The ring $R=k[x]/(x^2)$ is a hypersurface ring with unique maximal ideal $\mfm=(x)/(x^2)$, which is also the unique prime ideal and so $\Spec R=\{\mfm\}$ and $\dim R=0$. Since $R$ is not a regular ring, it follows that $\Sing R = \Spec R$. By~\Cref{thm:hyper-costrat}, it follows that $\rmS(R)$ is costratified and the colocalizing subcategories of $\rmS(R)$ stand in bijection with $\scP(\Sing R)=\{\varnothing,\Sing R\}$ and so $\Coloc (\rmS(R))=\{0,\rmS(R)\}$. One could of course obtain this conclusion by invoking~\Cref{lem:ps-cg-coloc}, since $R$ is a hypersurface ring of Krull dimension zero, so $\rmS(R)$ is pure-semisimple and the only localizing subcategories of $\rmS(R)$ are $0$ and $\rmS(R)$. The unique $\shom$-prime colocalizing subcategory of $\rmS(R)$ is $0$ and corresponds to the unique point $\mfm$ of $\Spec R$.
\end{ex}

\section{Schemes with hypersurface singularities}
\label{sec:schemes}%
In this section, we generalize~\Cref{thm:hyper-costrat} to schemes with hypersurface singularities by applying~\cite[Theorem 5.9]{Verasdanis23a}.

Let $X$ be a noetherian separated scheme with structure sheaf $\mcO_X$. We denote by $\QCoh X$ the abelian category of quasi-coherent $\mcO_X$-modules and by $\rmD(X)$ the derived category of $\QCoh X$. The derived category $\rmD(X)$ is a rigidly-compactly generated tensor-triangulated category with tensor product the left derived tensor product of complexes of $\mcO_X$-modules and unit $\mcO_X$ concentrated in degree zero. The subcategory of compact objects of $\rmD(X)$ is $\rmD^\mathrm{perf}(X)$ the subcategory of bounded complexes of locally free $\mcO_X$-modules up to quasi-isomorphism. There is a homeomorphism $\Spc(\rmD^{\mathrm{perf}}(X))\cong X$ and we will treat this as an equality. The \emph{singularity category} of $X$ is $\rmS(X)=\rmK_{\mathrm{ac}}(\Inj X)$ the homotopy category of acyclic complexes of injective quasi-coherent $\mcO_X$-modules, which is a compactly generated triangulated category by~\cite{Krause05}.

The results in the following discussion can be found in~\cite[Section 7]{Stevenson14}. Let $U$ be an open subset of $X$ and let $Z=X\setminus U$. Let $\rmD(X)_Z$ be the localizing subcategory of $\rmD(X)$ generated by those compact objects supported on $Z$. We denote by $\rmD(X)(U)$ the category $\rmD(X)/\rmD(X)_Z$. Then there is an equivalence $\rmD(X)(U)\simeq \rmD(U)$. There is an action of $\rmD(X)$ on $\rmS(X)$ that induces a support theory for objects of $\rmS(X)$ (and a cosupport theory; see~\cite{Verasdanis23a}). We denote by $\rmS(X)(U)$ the localizing subcategory of $\rmS(X)$ generated by those compact objects supported on $Z$. The category $\rmS(X)(U)$ is equivalent to $\rmS(U)$ and the action of $\rmD(X)$ on $\rmS(X)$ gives rise to an action of $\rmD(U)$ on $\rmS(U)$. If $\{U_i\cong \Spec R_i\}$ is an open affine cover of $X$, then the singular locus of $X$ is $\Sing X=\bigcup \Sing R_i$. 

A colocalizing subcategory $\scC$ of $\rmS(X)$ is called a \emph{colocalizing} $\shom$-\emph{submodule} if $[E,A]\in \scC,\, \forall E\in \rmD(X),\, \forall A\in \rmS(X)$, where $[E,-]\colon \rmS(X) \to \rmS(X)$ is the right adjoint of the action $E\ast -\colon \rmS(X) \to \rmS(X)$.
\begin{thm}\label{thm:costrat-schemes}%
Let $X$ be a noetherian separated scheme with hypersurface singularities. Then $\rmS(X)$ is costratified, i.e., there is a bijective correspondence between $\Sing X$ and the collection of colocalizing $\shom$-submodules of $\rmS(X)$ given by mapping a colocalizing $\shom$-submodule of $\rmS(X)$ to its cosupport.
\end{thm}

\begin{proof}
Let $X=\bigcup U_i$ be an open affine cover of $X$. Then each $U_i$ is isomorphic to $\Spec R_i$ for a commutative noetherian ring $R_i$ that is locally a hypersurface. By the above discussion, we have an action of $\rmD(U_i)=\rmD(R_i)$ on $\rmS(U_i)=\rmS(R_i)$ and since $R_i$ is locally a hypersurface, $\rmS(R_i)$ is costratified by~\Cref{thm:hyper-costrat}. A direct application of~\cite[Theorem 5.9]{Verasdanis23a} implies that $\rmS(X)$ is costratified.
\end{proof}

A colocalizing $\shom$-submodule $\scC$ of $\rmS(X)$ is called $\shom$-prime if, for all $E\in \rmD(X)$ and $A\in \rmS(X)$, if $[E,A]\in \scC$, then $[E,I_{\rmS(X)}]\in \scC$ or $A\in \scC$.
\begin{thm}\label{thm:schemes-hom-primes}%
Let $X$ be a noetherian separated scheme with hypersurface singularities. Then there is a bijective correspondence between points of $\Sing X$ and $\shom$-prime colocalizing submodules of $\rmS(X)$. A point $x \in \Sing X$ is associated with $\Ker ([g_x,-]\colon \rmS(X) \to \rmS(X))$, with the latter being equal to $\coloc([g_y,I_{\rmS(X)}] \mid y \neq x)$.
\end{thm}

\begin{proof}
By~\Cref{thm:costrat-schemes}, $\rmS(X)$ is costratified. The result now follows immediately from~\cite[Theorem 4.10]{Verasdanis23a}.
\end{proof}

\end{document}